\documentclass[11pt]{article}
\usepackage{amsmath, amssymb, amsfonts, amsthm}
\usepackage{enumitem}
\usepackage{multicol}
\allowdisplaybreaks
\def \Q {{\mathbb{Q}}}
\def \N {{\mathbb{N}}}
\def \Z {{\mathbb{Z}}}

\def \P {{\mathcal{P}}}

\newtheorem*{theorem*}{Theorem}
\newtheorem{theorem}{Theorem}
\newtheorem{cor}[theorem]{Corollary}

\newtheorem{lemma}[theorem]{Lemma}

\title{Irreducibility of $x^n-a$}
\author {Biswajit Koley,  A.Satyanarayana Reddy\thanks{The research of this author is supported by Matrics MTR/2019/001206  of SERB, India.} \\
Department of 
Mathematics, Shiv Nadar 
University, India-201314\\ (e-mail: 
bk140@snu.edu.in, satyanarayana.reddy@snu.edu.in).
  }
\date{}
\begin{document}
\maketitle
\begin{abstract}
A. Capelli gave a necessary and sufficient condition for the reducibility of $x^n-a$ over $\Q$. In this article, we are providing  an alternate elementary proof for the same. 
\end{abstract}
{\bf{AMS(2010)}}: 11R09, 12D05.\\

In this article, we present an elementary proof of a theorem about the irreducibility of $x^n-a$ over $\Q$. Vahlen\cite{vahlen} is the first mathematician who characterized the irreducibility conditions of $x^n-a$ over $\Q$. A. Capelli~\cite{capelli} extended this result to all fields of characteristic zero.  Later L. R\'edei~\cite{redei} proved this result  for  all fields of positive characteristic. But this theorem referred to as Capelli's theorem.
 
 \begin{theorem}[\cite{capelli}, \cite{vahlen}, \cite{redei}] \label{thm:vahlen}
 Let  $n\ge 2.$ A polynomial $x^n-a\in \Q[x]$ is reducible over $\Q$ if and only if 
 either $a=b^t$ for some $t|n, t>1$, or $4|n$ and $a=-4b^4$, for some $b\in \Q.$
 \end{theorem}
 
Since Theorem \ref{thm:vahlen} is true for arbitrary fields, all of the proofs are proved  by using field extensions except the proof given by Vahlen \cite{vahlen}. Vahlen assumes that the binomial $x^n-a$  is reducible and proves Theorem \ref{thm:vahlen} by using the properties of $n^{th}$ roots of unity and by comparing the coefficients on both sides of the following 
equation
$$x^n-a=(x^m+a_{m-1}x^{m-1}+\cdots+a_0)(x^{n-m}+b_{n-m-1}x^{n-m-1}+\cdots+b_0)$$ for some $m$, $0<m<n$.   Reader can consult (\cite{karpilovsky}, p.425) for a proof using field theory. We give a proof particularly over $\Q$ by using very little machinery. 

Let $f(x)=x^n-a$, $a=\frac{b}{c}\in \Q$ and $(b, c)=1$. Then $c^nf(x)=(cx)^n-c^{n-1}b\in \Z[x].$ Hence $x^n-a$ is reducible over $\Q$ if and only if $y^n-c^{n-1}b$ is reducible over $\Z$. It is, therefore, sufficient to consider $a\in \Z$ and throughout the article, by reducibility, we will mean reducible over $\Z$.

\begin{theorem}\label{thm:capelli}
 Let  $n\ge 2.$ A polynomial $x^n-a\in \Z[x]$ is reducible over $\Z$ if and only if 
 either $a=b^t$ for some $t|n, t>1$, or $4|n$ and $a=-4b^4$, for some $b\in \Z.$ 
\end{theorem}
 
The polynomial $x^n-1$ is a product of cyclotomic polynomials  and if $n=2^{n_1}u$ with $2\nmid u$, then  
\begin{equation*}
x^n+1=\prod\limits_{d|u}\Phi_{2^{n_1+1}d}(x).
\end{equation*}
Therefore, from now onwards we assume that $a>1,$ if not specified, and 
check the reducibility of the polynomial $x^n\pm a$ for $n\ge 2$.  If there exists a prime $p$ such that $p|a$ but $p^2\nmid a,$ then $x^n\pm a$ is irreducible by Eisenstein's criterion. In other words, if $a=p_1^{a_1}p_2^{a_2}\cdots p_k^{a_k}$ is the prime factorization of $a$ and $x^n\pm a$ is reducible, then
$a_i\ge 2$ for every $i\in \{1,2,\ldots,k\}.$  More generally, 
 
\begin{lemma}\label{lem:gcdge2}
 Let $n\ge 2$, $a=p_1^{a_1}p_2^{a_2}\cdots p_k^{a_k}$ be the prime factorization of $a$ and let $x^n\pm a$ be reducible. Then $\gcd(a_1,a_2,\ldots,a_k)\ge 2$ and $\gcd(\gcd(a_1,a_2,\ldots,a_k),n)>1.$ 
\end{lemma}
\begin{proof}
 We prove the result by induction on $k=\omega(a)$, the number of distinct prime divisors of $a.$  The roots of $x^n\pm a$ are of the form $a^{1/n}(\mp \zeta_n)^e$, where $e\in \Z$ and  $\zeta_n$ is a primitive $n^{\text{th}}$ root of unity. Since the proof barely depends upon the sign of roots, we restrict to the case $x^n-a$. Let $f(x)$ be a proper factor of $x^n-a,$ where  $\deg(f)=s<n$. If $f(0)=\pm d$, then $\pm d=a^{s/n}\zeta_n^w$ for some $w\in \Z$. 

Let $k=1$ and $a=p_1^{a_1}$. From Eisenstein's criterion, $a_1\ge 2$. If $d=p_1^{\alpha},$ then $d^n=a^s$ gives, $\alpha n=a_1s$. Since $a_1\ge 2$ and $s<n$, we deduce that $(a_1, n)>1$. 

Let $k=2$ and $a=p_1^{a_1}p_2^{a_2}$. From $d^n=a^s,$ let $d=p_1^{d_1}p_2^{d_2}$ be the prime factorization of $d$. Then $nd_1=a_1s, nd_2=a_2s$ would give $d_1a_2=d_2a_1$. If $(a_1, a_2)=1$, then $d_1=a_1c$ for some $c|d_2$ and $nc=s<n$ is a contradicton. Thus, $m=(a_1, a_2)\ge 2$. Next we need to show that $(n,m)>1$. Suppose $a_1=mb_1, a_2=mb_2$ with $(b_1,b_2)=1.$ From $d_1a_2=d_2a_1$, we deduce that  $d_1b_2=d_2 b_1.$ Then $d_1=b_1 r$ for some $r|d_2.$ If $(n,m)=1$, then $nd_1=a_1s=mb_1s$ will give $n|s$, a contradiction. Hence, $(n,m)>1$.

Suppose the result is true for some $k\ge 2$. Thus, if $a=p_1^{a_1}\cdots p_k^{a_k}$, then $(a_1, \ldots, a_k)=u>1$ and $(u,n)>1$. To show that the result is true for $k+1$. Let $a=p_1^{a_1}\cdots p_k^{a_k}p_{k+1}^{a_{k+1}}=a_1^up_{k+1}^{a_{k+1}}$, where $a_1=p_1^{w_1}\cdots p_k^{w_k}$ and $(w_1, \ldots, w_k)=1$. From $d^n=a^s$, we can write $d$ as $d=b_1^vp_{k+1}^{d_{k+1}},$ where $b_1=p_1^{v_1}\cdots p_k^{v_k},$ $(v_1, \ldots, v_k)=1$. From the fundamental theorem of arithmetic, $d_{k+1}n=a_{k+1}s$ and $b_1=a_1,$ $vn=us$. That is $d_{k+1}u=a_{k+1}v$. If $(u,a_{k+1})=1,$ then $d_{k+1}=a_{k+1}h$ for some $h|v,$ and $na_{k+1}h=a_{k+1}s$ implies $n|s$. This contradicts the fact that $s<n$. Thus, $(u,a_{k+1})=m>1$. 

To show that $(n,m)>1$. Let $u=mu_1, a_{k+1}=ma_{k+1}',$ where $(u_1, a_{k+1}')=1$. From $d_{k+1}u=a_{k+1}v$, we get $u_1d_{k+1}=va_{k+1}'$. Since $(u_1,a_{k+1}')=1$, $d_{k+1}=a_{k+1}'t$ for some $t|v$. On the other hand, $nd_1=a_1s, nd_{k+1}=a_{k+1}s$ would imply $a_1d_{k+1}=d_1a_{k+1}$. If $a_1=ua_1'=mu_1a_1',$ then $a_1d_{k+1}=d_1a_{k+1}$ implies  $d_1=u_1a_1't$. Using this in  $nd_1=a_1s$, we have $nt=ms$. If $(n,m)=1$, then $n|s$ is a contradiction. Thus, $(n,m)>1$. 
By induction principle, the result is true for every $k\ge 1.$
\end{proof}

In other words, if $x^n\pm a$ is reducible, then $a$ has to be of the form $b^m,$  where $(n,m)>1$ and $m\ge 2$. With a rearrangement in powers, we can say

\begin{cor}\label{cor:restrictedb}
Let $n\ge 2$ and $x^n\pm a $ be reducible over $\Z$. Then $ a= b^m$ for some $m\ge 2, m|n$, and $b$ is either a prime number or $b=(p_1^{b_1}p_2^{b_2}\cdots p_k^{b_k})^d,$ where $k\ge 2, (b_1, b_2, \ldots, b_k)=1$ and $(d, n)=1.$   
\end{cor}
Suppose $f(x)=x^{25}\pm 6^8$. Then $b=6,$ $m=8$, and  $(8,25)=1$ implies that the polynomial $x^{25}\pm 6^8$ is irreducible by Corollary \ref{cor:restrictedb}. Let $g(x)=x^{25}\pm (243)^2$. If we consider $b=243$, then  $(2,25)=1$ imply that $x^{25}\pm (243)^2$ is irreducible. But $x^5\pm 9|g(x)$. The reason is, $b=243$ is not as in Corollary \ref{cor:restrictedb}. Since $243=3^5,$  $b$  will be $3^2$ and $m=5$ so that $m|n$. Because of this reason, we will say

A positive integer `{\em $b$ has the property $\P$'} if $b$ is in the form as given in Corollary \ref{cor:restrictedb}.

\begin{lemma}\label{lem:m|n}
Let $m\ge 2, m|n,$ and $b$ has the property $\P$. Then $x^n\pm b^m$ is reducible except possibly for $x^n+b^{2^r}, r\ge 1$. 
\end{lemma}

\begin{proof} If $m|n$, then 
\begin{equation*}
x^n-b^m=\left(x^{n/m}-b\right)\left(x^{n(m-1)/m}+x^{n(m-2)/m}b+\cdots+x^{n/m}b^{m-2}+b^{m-1}\right).
\end{equation*}
Let $m=2^{r}m_1$, where $2\nmid m_1$ and $r\ge 0$. Then
\begin{align}
x^n+b^m&= \prod\limits_{d|m_1}b^{2^r\varphi(d)}\Phi_{2^{r+1}d}\left(\frac{x^{n/m}}{b}\right), \notag
\end{align}
where $b^{2^r\varphi(d)}\Phi_{2^{r+1}d}\left(\frac{x^{n/m}}{b}\right)\in \Z[x]$ and $\varphi$ is the Euler totient function.  
\end{proof}

Lemma \ref{lem:m|n} is true even if $b$ does not have the property $\P$. If $m=2^r\ge 2, m|n,$ and $b$ has the property $\P$, then the reducibility condition of $x^n+b^m$ completes the proof of Theorem \ref{thm:capelli}.

Selmer(\cite{selmer}, p.298) made the following observation. Let $g(x)\in \Z[x]$ be an arbitrary irreducible polynomial of degree $n$. If $g(x^2)$ is reducible, then, using the fact that $\Z[x]$ is a unique factorization domain, we get $$g(x^2)=(-1)^ng_1(x)g_1(-x),$$  where $g_1(x)$ is an irreducible polynomial in $\Z[x]$. Thus, if  $g_1(x)=a_nx^n+a_{n-1}x^{n-1}+\cdots+a_1x+a_0$, then 
\begin{equation*}
g(x^2)=(a_nx^n+a_{n-2}x^{n-2}+\cdots+a_0)^2-(a_{n-1}x^{n-1}+\cdots+a_1x)^2.
\end{equation*}
Let $k$ be an odd integer. Then $g(k^2)\equiv g(1)\pmod{4}$. Since the right hand side of the last equation is the difference between the two squares, $g(k^2)\equiv 0, \pm 1\pmod{4}$. Combining all of these, one can conclude that 

\begin{lemma}\label{general}
Let $g(x)\in \Z[x]$ be an irreducible polynomial. 
\begin{enumerate}[label=(\alph*)]
\item\label{gpart1} If $g(k^2)\equiv 2\pmod{4}$ for an odd integer $k$, then $g(x^2)$ is irreducible over $\Z$. 
\item\label{gpart2} If $g(x^2)$ is reducible, then there are unique (up to sign) polynomials $f_1(x)$ and $f_2(x)$ such that $g(x^2)=f_1(x)^2-f_2(x)^2$. Furthermore, in this case, we can write $f_1(x)=h_1(x^2)$ and $f_2(x)=xh_2(x^2),$ where $h_1(x),h_2(x)\in \Z[x]$. 
\end{enumerate}
\end{lemma}
The proof of \ref{gpart2} follows from the fact  that $\Z[x]$ is a unique factorization domain. 

\begin{lemma}\label{capelli:lem3}
Let $m=2^r\ge 2$ and $n$ be an odd positive integer. If $b$ has the property $\P$, then $x^{2^in}+b^m$ is irreducible for every $i$, $0\le i\le r$.
\end{lemma}

\begin{proof}
We proceed by induction on $i$. If $i=0$ and $b$ has the property $\P$, then $f(x)=x^n+b^m$ is irreducible by Lemma \ref{lem:gcdge2}. If $i=1$, then $f(x)=x^n+b^m$ is irreducible, and if $f(x^2)=x^{2n}+b^m=(x^n+b^{m/2})^2-2b^{m/2}x^n$ is reducible, from Lemma \ref{general}, $2b^{m/2}x^n$ has to be of the form $x^2h(x^2)^2$ for some $h(x)\in \Z[x]$. Since $n$ is odd, this is not possible and hence $f(x^2)$ is irreducible. 

Suppose the result is true for some $i$, $0\le i\le r$ and we will show that it is true for $i+1\le r$. So, $f(x)=x^{2^in}+b^m$ is irreducible for some $i\le r$. From Lemma \ref{general}, if  
\begin{equation*}
f(x^2)=x^{2^{i+1}n}+b^m=(x^{2^in}+b^{m/2})^2-2b^{m/2}x^{2^in}
\end{equation*}
is reducible, then $2b^{m/2}x^{2^in}$ has to be of the form $(xg(x^2))^2$ for some $g\in \Z[x]$. This is possible only when $m=2$ and $b=2^{\alpha}b_1^2,$ where $\alpha, b_1$ are odd positive integers. That is $r=1$ and hence $i=0$. We have already seen that $f(x^2)$ is irreducible in this case. Therefore, by the induction principle, $x^{2^in}+b^{2^r}$ is irreducible for every $i\le r$.  
\end{proof}

\begin{lemma}
Let $m=2^r\ge 2$ and let $b$ be an odd integer which has the property $\P$. If $m|n,$ then $x^n+b^m$ is irreducible. 
\end{lemma}

\begin{proof}
Let $n=mt$. If $t$ is odd, then by Lemma~\ref{capelli:lem3}, $x^n+b^m$ is irreducible. Let $t=2^{t_1}u$, $t_1\ge 1$ and $u$ is odd. From Lemma \ref{capelli:lem3}, $g(x)=x^{mu}+b^m$ is irreducible. Since $b$ is odd, $g(k^2)\equiv 2\pmod{4}$ for any odd integer $k$. Applying Lemma \ref{general} repeatedly to $g(x)$, the result follows. 
\end{proof}

\begin{cor}\label{cor:tbeven}
Let $m=2^r\ge 2,$ $m|n,$ and $b$ has the property $\P$. If  $x^n+b^m$ is reducible, then both $b$ and $\frac{n}{m}$ are even integers.  
\end{cor}

\begin{lemma}\label{lastlem}
Let  $t,r\in \N$ and $b$ has the property $\P$. Then $x^{2^rt}+b^{2^r}$ is reducible if and only if $t$ is even, $r=1$, and $b=2d^2$ for some $d\in \N$. 
\end{lemma}

\begin{proof}
If $t$ is even, $r=1$ ,and $b=2d^2,$ then $$x^{2t}+4d^4=(x^t+2d^2)^2-4d^2x^t=(x^{t}-2dx^{t/2}+2d^2)(x^{t}+2dx^{t/2}+2d^2).$$
Conversely, let $g(x)=x^{2^rt}+b^{2^r}$ is reducible. By Corollary \ref{cor:tbeven}, both $t$ and $b$ are even integers. Let $t=2^{t_1}u, b=2^{b_1}v,$ where $u,v$ are odd integers and $t_1, b_1\ge 1$. By Lemma \ref{capelli:lem3}, the polynomial $h(x)=x^{2^ru}+b^{2^r}$ is irreducible. Since $g(x)=h(x^{2^{t_1}})$ is reducible, there is some $i$, $1\le i \le t_1$ such that $h(x^{2^{i-1}})$ is irreducible and 

\begin{equation*}
h(x^{2^i})=x^{2^{r+i}u}+2^{2^rb_1}v^{2^r}=(x^{2^{r+i-1}u}+2^{2^{r-1}b_1}v^{2^{r-1}})^2-2^{2^{r-1}b_1+1}v^{2^{r-1}}x^{2^{r+i-1}u}
\end{equation*}
is reducible. From uniqueness property of Lemma \ref{general}, $2^{2^{r-1}b_1+1}v^{2^{r-1}}x^{2^{r+i-1}u}$ has to be of the form $(xl(x^2))^2$ for some $l(x)\in \Z[x]$. This is possible only when $r=1$ and $2^{2^{r-1}b_1+1}v^{2^{r-1}}$ is a perfect square. Hence, $2^{b_1+1}v=c^2$ would imply $b=2\left(\frac{c}{2}\right)^2$ with $\frac{c}{2}\in \N$.  
\end{proof}

 Proof of Theorem~\ref{thm:capelli} and hence Theorem~\ref{thm:vahlen} follows from Lemma~\ref{lem:gcdge2} and \ref{lem:m|n}, \ref{lastlem}.\\ 

{\bf Acknowledgement.} We would like to thank the referee for valuable comments.

\end{document}